\documentclass[12pt]{amsart}
\usepackage{amssymb}
\usepackage{amsfonts}
\usepackage{amsmath}
\usepackage{mathtools}
\usepackage[cmyk]{xcolor}
\usepackage[margin=1.1in]{geometry}

\DeclarePairedDelimiter{\ceil}{\lceil}{\rceil}

\newtheorem{theorem}{Theorem}[section]
\newtheorem{lemma}[theorem]{Lemma}
\newtheorem{corollary}[theorem]{Corollary}
\newtheorem{proposition}[theorem]{Proposition}
\newtheorem{remark0}[theorem]{Remark}
\newtheorem{example0}[theorem]{Example}
\newtheorem{definition}[theorem]{Definition}

\newenvironment{example}{\begin{example0}\rm}{\end{example0}}
\newenvironment{remark}{\begin{remark0}\rm}{\end{remark0}}

\newcommand{\propref}[1]{Proposition~\ref{#1}}
\newcommand{\thmref}[1]{Theorem~\ref{#1}}
\newcommand{\lemref}[1]{Lemma~\ref{#1}}

\newcommand{\exref}[1]{Example~\ref{#1}}

\newcommand{\remref}[1]{Remark~\ref{#1}}

\def\res{{\mathbf{k}}}

\def\HF{{\operatorname{H\!F}}}
\def\HP{\operatorname{H\!P}}

\def\deg{\operatorname{deg}}
\def\Max{\operatorname{Max}}
\def\dim{\operatorname{dim}}

\def\img{\operatorname{Im}}
\def\ker{\operatorname{Ker}}

\begin{document}
\title[Sumsets and monomial projective curves]{{\bf
Sumsets and projective curves}}
\author[J. Elias]{J. Elias ${}^{*}$}
\thanks{${}^{*}$
Partially supported by PID2019-104844GB-I00}
%; 05B10}
\address{Joan Elias
\newline \indent Departament de Matem\`{a}tiques i Inform\`{a}tica
\newline \indent Universitat de Barcelona (UB)
\newline \indent Gran Via 585, 08007
Barcelona, Spain}  \email{{\tt elias@ub.edu}}

\begin{abstract}
The aim of this note is to  exploit a new relationship  between  additive combinatorics and the geometry of
monomial projective curves.
We associate to a finite set of non-negative integers $A=\{a_1,\cdots, a_n\}$  a monomial projective curve $C_A\subset \mathbb P^{n-1}_{\res}$ such that the Hilbert function of $C_A$ and
the cardinalities of $sA:=\{a_{i_1}+\cdots+a_{i_s}\mid 1\le i_1\le \cdots \le i_s\le n\}$ agree.
The singularities of $C_A$ determines the asymptotic behaviour of $|sA|$, equivalently the Hilbert polynomial of $C_A$, and the asymptotic structure of $sA$.
We show that some additive inverse problems can be translate to the rigidity of Hilbert polynomials and we improve an upper bound of the Castelnuovo-Mumford regularity of monomial projective curves by using results of additive combinatorics.
\end{abstract}

\maketitle

%%%%%%%%%%%%%%%%%%%%%%%%%%%%%%%%%%%%%%%%%%%%%%%%%%%%%%%%%%%%%%%%%%%%%%%%%%%%%%%%%%%%%%%%%%%%%%%%
%%%%%%%%%%%%%%%%%%%%%%%%%%%%%%%%%%%%%%%%%%%%%%%%%%%%%%%%%%%%%%%%%%%%%%%%%%%%%%%%%%%%%%%%%%%%%%%%
%%%%%%%%%%%%%%%%%%%%%%%%%%%%%%%%%%%%%%%%%%%%%%%%%%%%%%%%%%%%%%%%%%%%%%%%%%%%%%%%%%%%%%%%%%%%%%%%
%%%%%%%%%%%%%%%%%%%%%%%%%%%%%%%%%%%%%%%%%%%%%%%%%%%%%%%%%%%%%%%%%%%%%%%%%%%%%%%%%%%%%%%%%%%%%%%%
%%%%%%%%%%%%%%%%%%%%%%%%%%%%%%%%%%%%%%%%%%%%%%%%%%%%%%%%%%%%%%%%%%%%%%%%%%%%%%%%%%%%%%%%%%%%%%%%
%%%%%%%%%%%%%%%%%%%%%%%%%%%%%%%%%%%%%%%%%%%%%%%%%%%%%%%%%%%%%%%%%%%%%%%%%%%%%%%%%%%%%%%%%%%%%%%%
%\bigskip
\section{Introduction }

Let $A=\{a_1,\cdots, a_n\}$, $n\ge 2$, be a set of different non-negative integers; we assume that
$a_1<\cdots <a_n$.
Given a non-negative integer $s\ge 1$ the $s$-fold iterated sumset of $A$ is
$$
sA=\{a_{i_1}+\cdots+a_{i_s}\mid 1\le i_1\le \cdots \le i_s\le n\},
$$
we set  $0A=\{0\}$; notice that $1 A=A$.

Following Nathanson, a direct problem in  additive combinatorics is a problem in which we try to determine the structure and properties of $|sA|$, $s\ge 0$, when the set $A$ is known.
On the other hand, an inverse problem in additive combinatorics   is a problem in which we attempt to deduce properties of $A$ from properties of $sA$, $s\ge 0$, \cite{Nat96}.

The aim of this paper is to establish and to study a bridge between  additive combinatorics and the geometry of
monomial projective curves.
We argue back and forth: we use  results of monomial projective curves to recover or to improve  results of additive combinatorics and vice versa,
see \thmref{lowerbound2} and \thmref{bermejo}.
In particular, we show that some inverse problems can be translate in terms of the rigidity of  Hilbert polynomials, see Section 4 and \cite{Eli90}.

In this paper, we have selected some significative results of  the geometry of monomial projective curves and additive combinatorics; there are a huge number of results and properties of both areas to link that we
will consider elsewhere, see \cite{CEM21}.

\medskip
The contents of the paper is the following.
In the second section, following \cite{EM20}, we attach to the set $A$ a monomial projective curve $C_A\subset \mathbb P_{\res}^{n-1}$. The  Hilbert function of $C_A$ and
the cardinalities of $sA$, $s\ge 0$, agree.
Some previous results can be found in \cite{Kho95}.

In the section three we use the data provided by the singularities of $C_A$ to determine  the asymptotic behaviour of $|sA|$, equivalently the Hilbert polynomial of $C_A$, \propref{HilbertPol}.
As  a consequence we can describe the asymptotic decomposition of $sA$ of the so-called fundamental result of additive combinatorics,  \propref{Funda} and \propref{refinFun}.

The section 4 is devoted to recover, by considering generic hyperplane sections of $C_A$, some additive inverse results and to link them with rigid polynomials and rigid properties, \propref{lowerbound}, \thmref{lowerbound2}.
We finish the paper improving an upper bound on the Castelnuovo-Mumford regularity of $C_A$ established in \cite{BGG17} by using a result of L.V. Lev on the growth of $|sA|$, \thmref{bermejo}.

For the basic results on algebra, algebraic geometry or additive number theory we will   use:
 \cite{BH97},  \cite{Har97},  \cite{Nat96}.
The computations of this paper are performed by using Singular, \cite{DGPS}.

\medskip
\begin{center}
  {\sc Notations}
\end{center}

In this paper $\res$ is an arbitrary infinite field.
  Let $R=\sum_{i\ge 0} R_i$ be an standard $\res=R_0$ algebra, i.e. $R=\res[R_1]$.
  We denote by $\HF_R$ the Hilbert function of $R$, i.e. $\HF_R(i)=\dim_{\res} R_i$ for all $i\ge 0$.
  It is known that there exists a rational coefficient  polynomial $\HP_R$, Hilbert polynomial of $R$,
  such that $\HP(i)=\HF(i)$ for $i\gg 0$.

Given a set $B$ of non-negative integers $b_1,\dots ,b_n$ we denote by
$\langle b_1,\cdots ,b_n\rangle$ the sub-semigroup of $\mathbb N$ generated by $B$.
Given a multi-index $\alpha=(\alpha_1,\cdots,\alpha_n) \in \mathbb N^n$ we define its total order by
$|\alpha|=\sum_{i=1}^n \alpha_i$ and the total order with respect to $A$ by
$|\alpha|_A=\sum_{i=1}^n a_i \alpha_i$.

%%%%%%%%%%%%%%%%%%%%%%%%%%%%%%%%%%%%%%%%%%%%%%%%%%%%%%%%%%%%%%%%%%%%%%%%%%%%%%%%%%%%%%%%%%%%%%%%
%%%%%%%%%%%%%%%%%%%%%%%%%%%%%%%%%%%%%%%%%%%%%%%%%%%%%%%%%%%%%%%%%%%%%%%%%%%%%%%%%%%%%%%%%%%%%%%%
%%%%%%%%%%%%%%%%%%%%%%%%%%%%%%%%%%%%%%%%%%%%%%%%%%%%%%%%%%%%%%%%%%%%%%%%%%%%%%%%%%%%%%%%%%%%%%%%
%%%%%%%%%%%%%%%%%%%%%%%%%%%%%%%%%%%%%%%%%%%%%%%%%%%%%%%%%%%%%%%%%%%%%%%%%%%%%%%%%%%%%%%%%%%%%%%%
%%%%%%%%%%%%%%%%%%%%%%%%%%%%%%%%%%%%%%%%%%%%%%%%%%%%%%%%%%%%%%%%%%%%%%%%%%%%%%%%%%%%%%%%%%%%%%%%
%%%%%%%%%%%%%%%%%%%%%%%%%%%%%%%%%%%%%%%%%%%%%%%%%%%%%%%%%%%%%%%%%%%%%%%%%%%%%%%%%%%%%%%%%%%%%%%%
\medskip
\section{The bridge between additive number theory and projective curves}

We first show that we can consider several straight simplifications on the set $A$ and an easy property on the growth of $|sA|$, see \cite{Nat96},

%\medskip
\begin{lemma}
\label{1Reduction}
Given a set of non-negative integers $A=\{a_1,\cdots, a_n\}$, $n\ge 2$, with $a_1<\cdots <a_n$, it holds:
\begin{enumerate}
  \item In order to compute $|sA|$ we may assume that $a_1=0$ and $\gcd(a_2,\cdots, a_n)=1$,
   \item under the above conditions,  $|(s+1)A|\ge |sA|+n-1$ for all $s\ge 0$.
\end{enumerate}
\end{lemma}
\begin{proof}
\noindent
$(1)$ Let us consider $A'=\{0, (a_2-a_1)/d,\cdots, (a_n-a_1)/d\}$ where $d=\gcd(a_2-a_1,\cdots, a_n-a_1)$.
It is easy to see that $|sA| =|sA'|$ for all $s\ge 0$.

\noindent
$(2)$
Assume that $A$ satisfies the conditions of $(1)$.
Since the maximum of $sA$ is $sa_n$ we deduce that $s a_n+a_2,\cdots, sa_n+a_n\in (s+1)A\setminus sA$, so we get the claim:
$|(s+1)A|\ge |sA|+n-1$ for all $s\ge 0$.
\end{proof}

%\medskip
Given a general set of non-negative integers $A$, the associated set $A'$ of the proof of the previous Lemma, is called the normal form of $A$, see \cite {Nat96}.
From now on we assume that a set $A$ satisfies \lemref{1Reduction} $(1)$.

Next, we recall the key construction of \cite{EM20}.

\begin{definition}
  We denote by $R(A)$ the $\res$-subalgebra of $\res[t,w]$ generated by $t^{a_i} w$, $i=1,\dots, n$.
  We consider $\res[t,w]$ endowed with the grading defined by $\deg(t)=0$, $\deg(w)=1$.
\end{definition}

Let $\phi=\res[X_1,\cdots, X_n]\longrightarrow \res[t,w]$ the degree zero $\res$-algebra morphism defined by
$\phi(X_i)=t^{a_i} w$.
We have $\img(\phi)=R(A)$ and the homogeneous piece of degree $s$ of $R(A)$, i.e. $R(A)_s$, admits
the $\res$-basis
\begin{equation}
\label{basis}
  t^{\alpha}w^s, \quad\alpha\in sA.
\end{equation}

From this fact we get:

%\medskip
\begin{proposition}\cite[Section 2]{EM20}
For all $s\ge 0$ it holds $\HF_{R(A)}(s)=|sA|$.
\end{proposition}

In the following result a  system of generators of $\ker(\phi)$ is computed:

%\medskip
\begin{proposition}\cite[ Proposizione 2.2]{CN84}, \cite[Proposition 6.4]{EM20}
  \label{EM-R(A)}
  The kernel of $\phi$ is generated by the binomials $X^\alpha-X^\beta$, $\alpha, \beta \in \mathbb N^n$, such that   $|\alpha|=|\beta|$ and $|\alpha|_A=|\beta|_A$.
\end{proposition}

Next, we link $R(A)$ with a suitable monomial projective curve.

\begin{definition}
  Given a set $A=\{a_1=0,a_2,\cdots, a_n\}$  such that $a_1<\cdots < a_n$ and  $\gcd(a_2,\cdots, a_n)=1$ we consider the monomial curve $C_A$ of $\mathbb P_{\res}^{n-1}$ defined by the Kernel of
  $$
  \begin{array}{cccc}
    \psi & \res[X_1,\cdots, X_n] & \longrightarrow & \res[u,v] \\
         & X_i & \mapsto & u^{a_n-a_i}v^{a_i}
  \end{array}
  $$
\end{definition}

If we consider the standard grading of $\res[u,v]$ we get that $\ker(\psi)=I_{A}$ is a homogeneous ideal of
$\res[X_1,\cdots, X_n]$.
We denote by $\res[C_A]:=\res[X_1,\cdots,X_n]/\ker(\phi)$ the homogeneous coordinate ring of $C_A$.
We write $\HF_{C_A}=\HF_A$ and $\HP_{C_A}=\HP_A$.

%\medskip
\begin{proposition}
  \label{bridge}
  For all set $A=\{a_1=0,a_2,\cdots, a_n\}$ we have that $\ker (\phi)=I_{A}$.
  Hence $R(A)\cong \res[C_A]$ as graded $k$-algebras.
\end{proposition}
\begin{proof}
We first prove that $\ker (\phi)\subset I_A$.
Let's consider a  binomial  $X^\alpha-X^\beta$, $\alpha, \beta \in \mathbb N^n$, with   $|\alpha|=|\beta|$ and $|\alpha|_A=|\beta|_A$.
Then
$$
\psi(X^\alpha-X^\beta)=u^{a_n |\alpha|-|\alpha|_A}v^{|\alpha|_A}-u^{a_n |\beta|-|\beta|_A}v^{|\beta|_A}=0,
$$
by \propref{EM-R(A)} we get that $\ker(\phi)\subset \ker (\psi)=I_A$.

Next, we  prove that $I_A\subset \ker (\phi)$.
Let $F\in I_A$ be a polynomial, so
$$
F(u^{a_n}, u^{a_n-a_2}v^{a_2},\cdots, u^{a_n-a_{n-1}}v^{a_{n-1}}, v^{a_n})=0.
$$
If  $X^\alpha$, $\alpha\in \mathbb N^n$, is a monomial of $F$ then
$$
X^\alpha(u^{a_n}, u^{a_n-a_2}v^{a_2},\cdots, u^{a_n-a_{n-1}}v^{a_{n-1}}, v^{a_n})=
u^{a_n|\alpha|-|\alpha|_A}v^{|\alpha|_A}
$$
Hence we may assume that $F$ is a homogeneous polynomial
$$
F=\sum_{i=1}^d \lambda_i X^{\alpha_i}
$$
such that $|\alpha_i|_A=c$, $a_n|\alpha_i|=c+d$ and $\lambda_i\in \res\setminus \{0\}$.

Since $F(u^{a_n}, u^{a_n-a_2}v^{a_2},\cdots, u^{a_n-a_{n-1}}v^{a_{n-1}}, v^{a_n})=0$ we deduce that
$\sum_{i=1}^d \lambda_i =0,$
so
$$
F=\sum_{i=1}^{d-1} \lambda_i ( X^{\alpha_i}- X^{\alpha_d}) \in \ker (\phi).
$$
\end{proof}

\begin{remark}
\label{BB}
We write $\mathcal B_A=\frac{\res[C_A]}{X_1 \res[C_A]}$, notice that $\mathcal B_A$ is a  graded algebra of dimension one since the coset of $X_1$ is a non-zero divisor of $\res[C_A]$;
 $\mathcal B_A$ is the homogeneous coordinate ring of the hyperplane section of $C_A$ defined by $X_1=0$.
Both algebras $\res[C_A]$ and $\mathcal B_A$ are standard algebras, i.e. generated by their homogeneous pieces of degree one,  i.e. $\res[C_A]_1$ and
$(\mathcal B_A)_1$, respectively.
In general $\mathcal B_A$ is non Cohen-Macaulay as the classic example of Macaulay shows, see \exref{mac-ex}.
\end{remark}

\begin{example}
 \label{calculCA}
Let us consider the set $A=\{0,2,4,5,7\}$.
The associated monomial  curve $C_A$ is defined by the parameterization
$(u,v)\mapsto (u^7,u^5v^2,u^3v^4,u^2v^5,v^7)$.
Then the defining ideal of $C_A$ is minimally generated by
$x_2^2-x_1 x_3, x_2 x_4-x_1 x_5, x_3 x_4-x_2 x_5, x_2 x_3^2-x_1 x_4^2, x_3^3-x_1 x_4 x_5, x_4^3-x_3^2 x_5
$, \cite{DGPS}.
The Hilbert function of $C_A$ is $\HF_A=\{1,5,12,19, 26, 33, \cdots\}$ and the Hilbert polynomial $\HP_A(s)=7s-2$.
\end{example}

%%%%%%%%%%%%%%%%%%%%%%%%%%%%%%%%%%%%%%%%%%%%%%%%%%%%%%%%%%%%%%%%%%%%%%%%%%%%%%%%%%%%%%%%%%%%%%%%
%%%%%%%%%%%%%%%%%%%%%%%%%%%%%%%%%%%%%%%%%%%%%%%%%%%%%%%%%%%%%%%%%%%%%%%%%%%%%%%%%%%%%%%%%%%%%%%%
%%%%%%%%%%%%%%%%%%%%%%%%%%%%%%%%%%%%%%%%%%%%%%%%%%%%%%%%%%%%%%%%%%%%%%%%%%%%%%%%%%%%%%%%%%%%%%%%
%%%%%%%%%%%%%%%%%%%%%%%%%%%%%%%%%%%%%%%%%%%%%%%%%%%%%%%%%%%%%%%%%%%%%%%%%%%%%%%%%%%%%%%%%%%%%%%%
%%%%%%%%%%%%%%%%%%%%%%%%%%%%%%%%%%%%%%%%%%%%%%%%%%%%%%%%%%%%%%%%%%%%%%%%%%%%%%%%%%%%%%%%%%%%%%%%
%%%%%%%%%%%%%%%%%%%%%%%%%%%%%%%%%%%%%%%%%%%%%%%%%%%%%%%%%%%%%%%%%%%%%%%%%%%%%%%%%%%%%%%%%%%%%%%%
\medskip
\section{Sumsets and monomial projective curves}

We first recall some well known results on curves applied to the projective curve $C_A$, \cite{Har97}.
The monomial projective curve $C_A$ is rational with two eventually singular points
$P_1=(1,0,\cdots,0),  P_2=(0,\cdots,0,1)\in\mathbb P^{n-1}_{\res}$.
In the affine open neighborhood  $X_1=1$ of  $P_1$ the curve $C_A$  is defined by the parameterization
$v\mapsto (v^{a_2},\cdots, v^{a_n})$;  and in the open affine neighborhood $X_n=1$ of $P_2$ the curve
is defined by the parameterization  $u\mapsto (u^{a_n},  u^{a_n-a_2}, \cdots ,u^{a_n-a_{n-1}})$.
The point $P_1$ is non-singular iff $a_2=1$ and $P_2$ is non-singular iff $a_n-a_{n-1}=1$.

We denote by $p_a(C_A)$ the arithmetic genus of $C_A$, i.e.
$$
\HP_A(0)=1-p_a(C_A).
$$
Since $C_A$  is rational  its geometric genus is zero and
$$
p_a(C_A)=\sum_{P\in Sing(C_A)} \delta(C_A,P),
$$
where  $\delta(C_A,P)$ is the singularity order of $P\in Sing(C_A)$, i.e.
$$
\delta(C_A,P)=\dim_{\res}\frac{\overline{\mathcal O_{C_A,P}}}{\mathcal O_{C_A,P}}
$$
where the over-line stands for the integral closure of ${\mathcal O_{C_A,P}}$ in its field of fractions.
Summarizing, we get
$$
\HP_A(0)=1-\delta(C_A,P_1)-\delta(C_A,P_2).
$$
Since $C_A$ is a monomial curve in an affine neighbourhood of $P_1$ (resp. $P_2$) we have
$$
\delta(C_A,P_1)=Card(\mathbb N\setminus \langle a_2,\cdots, a_n \rangle)
$$
and
$$
\delta(C_A,P_2)=Card(\mathbb N\setminus \langle a_n-a_{n-1},\cdots, a_n-a_2,a_n \rangle).
$$

We know that the Hilbert polynomial $\HP_A(s)$ and the Hilbert function $\HF_A(s)$ agree for $s\gg 0$.
The first integer $s_0$ such that $\HF_A(s)=\HP_A(s)$ for all $s\ge s_0$ is called the regularity of the Hilbert function and it is denoted by $r(C_A)$.

The Castelnuovo-Mumford regularity $reg(C_A)$ of $C_A$, see \cite{Eis95}, for monomial projective curves is upper bounded in terms of the set $A$.
From \cite[Proposition 5.5]{Lvo96}, see also \cite{HHS10},
$$
reg(C_A)\le \rho(A):=1+\Max \{ (a_i-a_{i-1})+(a_j-a_{j-1})   ;2\le i<j\le n\}
$$
since $r(C_A)\le reg(C_A)$ we get that  $\HF_A(s)=\HP_A(s)$
for all $s\ge \rho(A)$.

Notice that $\rho_A\le a_n-n+3$.
This inequality can be deduced from the upper bound of the Castelnuovo-Mumford regularity conjectured
by Eisenbud and Goto and proved  by Gruson-Lazarsfeld-Peskine in the case of smooth curves, \cite{GLP83}.
If $C_A$ is non-singular then we have a better upper bound of the Castelnuovo-Mumford regularity, \cite[ Theorem 2.7]{HHS10},
$$
reg(C_A)\le 1+\Max \{ (a_i-a_{i-1})   ;2\le i<j\le n\}.
$$

The following result describes the asymptotic behaviour of $|sA|$, see \cite{GMV21}, \cite{Kho95}, \cite{NR02}.

%\medskip
\begin{proposition}
  \label{HilbertPol}
  Given a set $A=\{a_1=0, a_2,\cdots, a_n\}$ of integers such that $a_0< a_1< \cdots <a_n$ with $\gcd(a_2,\cdots, a_n)=1$ it holds
  $$
  |sA|=\HF_A(s)=s a_n +1- \delta(C_A,P_1)-\delta(C_A,P_2)
  $$
  for all  $s\ge \rho(A)$.
\end{proposition}
\begin{proof}
We know that $C_A$ is a degree $a_n$ projective curve, so
$$
\HP_A(s)=s a_n+ \HP_A(0)= s a_n +1- \delta(C_A,P_1)-\delta(C_A,P_2).
$$
Since $\HF_A(s)=\HP_A(s)$ for all $s\ge \rho(A)$ and we know that $|sA|=\HF_A(s)$ for all $s\ge0$, we get the claim.
\end{proof}

%\medskip
\begin{corollary}
\label{multB}
$\mathcal B_A$ is a one dimensional standard graded algebra of multiplicity $a_n$.
\end{corollary}
\begin{proof}
Since $X_1$ is a non-zero divisor of $\res[C_A]$, \remref{BB}, we get the claim from the last proposition.
\end{proof}

%\medskip
The often called fundamental result of additive combinatorics claims:

%\medskip
\begin{proposition}\cite[Theorem 1.1]{Nat96}
\label{Funda}
Given a set $A=\{a_1=0, a_2,\cdots, a_n\}$ of integers such that $a_0< a_1< \cdots <a_n$ with $\gcd(a_2,\cdots, a_n)=1$, there exists a positive integer $\sigma$, non-negative integers $c_1, c_2$ and finite sets $C_1\subset [0,c_1-2]$ and $C_2\subset [0,c_2-2]$ such that
$$sA=C_1\sqcup [c_1,s a_n-c_2]\sqcup (\{s a_n\}-C_2)$$
 for all $s\ge \sigma$.
\end{proposition}

Notice that from the above identity of sets we deduce
$$
|sA|=a_n s + 1-(c_1-|C_1|+c_2-|C_2|)
$$
for $s\ge \sigma$.
From \propref{HilbertPol} we get that
$$
\delta(C_A,P_1)+\delta(C_A,P_2)=c_1-|C_1|+c_2-|C_2|.
$$

Let $\Gamma_1$ be the semigroup generated by $a_1,\cdots, a_n$ and let $\Gamma_2$ be
the semigroup generated by $a_n-a_{n-1},\cdots, a_n-a_2, a_n$.
Notice that $\Gamma_i$ is the semigroup of the  curve singularity germ  $(C_A, P_i)$, $i=1,2$.

Next, we determine the set $C_i$ and the integer $c_i$, $i=1,2$, in terms of the eventual singular points of the
projective curve $C_A$.
Notice that $C_i=\emptyset$ iff $P_i$ is a non-singular point of $C_A$, $i=1,2$.

%\medskip
\begin{proposition}
 \label{refinFun}
Following  the notations of \propref{Funda}, we have that, $i=1,2$,
$$
\delta(C_A,P_i)=c_i-|C_i|
$$
$c_i$ is the conductor of $\Gamma_i$ and
$C_i=\Gamma_i\cap [0,c_i-2]$.
\end{proposition}
\begin{proof}
 We only have to prove the result for $i=1$.
Notice that if $s\ge \Max\{\sigma, (c_1+c_2)/a_n\}$ then
$$
[c_1,c_1+a_2]\subset sA.
$$
Moreover, since $sA\subset (s+1)A$, $s\ge 1$, we have  for all $s\gg 0$ that
$$
[c_1,c_1+a_2]\subset  sA \cap [0,c_1+a_2]=\Gamma_1 \cap [0,c_1+a_2].
$$
From this we get that $c_1$ is the conductor of $\Gamma_1$ and that
$$
C_1=\Gamma_1\cap [0,c_1-2].
$$
\end{proof}

\begin{example}
 \label{PolCD}
We consider the set $A=\{0,2,4,5,7\}$ of \exref{calculCA}.
The decomposition of $5A$ is
$$
5A=\{0,2\}\sqcup [4,33]\sqcup \{35\}
$$
so $c_1=4$, $C_1=\{0,2\}$, $c_2=2$ and $C_2=\{0\}$.
In this case we have $\Gamma_1=\{0,2,4,5,\cdots\}$, $\Gamma_2=\{0,2,3,\cdots\}$ and
$\delta_1=2$, $\delta_1=1$.
\end{example}

%%%%%%%%%%%%%%%%%%%%%%%%%%%%%%%%%%%%%%%%%%%%%%%%%%%%%%%%%%%%%%%%%%%%%%%%%%%%%%%%%%%%%%%%%%%%%%%%
%%%%%%%%%%%%%%%%%%%%%%%%%%%%%%%%%%%%%%%%%%%%%%%%%%%%%%%%%%%%%%%%%%%%%%%%%%%%%%%%%%%%%%%%%%%%%%%%
%%%%%%%%%%%%%%%%%%%%%%%%%%%%%%%%%%%%%%%%%%%%%%%%%%%%%%%%%%%%%%%%%%%%%%%%%%%%%%%%%%%%%%%%%%%%%%%%
%%%%%%%%%%%%%%%%%%%%%%%%%%%%%%%%%%%%%%%%%%%%%%%%%%%%%%%%%%%%%%%%%%%%%%%%%%%%%%%%%%%%%%%%%%%%%%%%
%%%%%%%%%%%%%%%%%%%%%%%%%%%%%%%%%%%%%%%%%%%%%%%%%%%%%%%%%%%%%%%%%%%%%%%%%%%%%%%%%%%%%%%%%%%%%%%%
%%%%%%%%%%%%%%%%%%%%%%%%%%%%%%%%%%%%%%%%%%%%%%%%%%%%%%%%%%%%%%%%%%%%%%%%%%%%%%%%%%%%%%%%%%%%%%%%
\medskip
\section{Rigid Hilbert polynomials and additive inverse problems}

In this section we link the inverse problems  with
the rigidity of  Hilbert polynomials and functions, \cite{Eli90}, \cite{EV91}.
In particular, we will recover several upper and lower bounds of the function $|sA|$ from some properties of the Hilbert function of $C_A$.

\begin{definition}
Let $H:\mathbb N\longrightarrow \mathbb N$ be a numerical function  asymptotically polynomial, i.e. there exists a polynomial $p(T)\in \mathbb Z[T]$ such that $H(s)=p(s)$ for $s\gg0$.
Let $\mathcal C$ be a class of graded $\res$-algebras.
We say that $p(T)$ is a rigid polynomial for the class $\mathcal C$ if for all graded $\res$ algebra $D$ of $\mathcal C$ if
$\HP_D=p$ then $\HF_D=H$, see \cite{Eli90}.
\end{definition}

From \lemref{1Reduction} (2) we get:

%\medskip
\begin{proposition}\cite[Theorems 1.3]{Nat96}
\label{lowerbound}
Given a set $A=\{a_1=0, a_2,\cdots, a_n\}$ of integers such that $a_0< a_1< \cdots <a_n$ with $\gcd(a_2,\cdots, a_n)=1$, for all $s\ge 0$ it holds
$$
 s(n-1)+1 \le |sA| \le \binom{s+n-1}{s}.
$$
\end{proposition}
\begin{proof}
  From \lemref{1Reduction} (2) we deduce the left hand inequality.
  The right hand inequality follows from \propref{bridge}.
  \end{proof}

In the next result we get \cite[Theorems 1.2, 1.6 and 1.8]{Nat96}; in particular  we prove that $p(T)=(n-1)T+1$ is a rigid polynomial for the class of $\res[C_A]$ algebras and that the condition
$|sA|= s(n-1)+1$, for some $s\ge 2$,  is a  rigid property, i.e. determines the whole Hilbert function, see \cite{EV91}.

%\medskip
\begin{theorem}\cite[Theorems 1.2, 1.6, 1.8]{Nat96},
\label{lowerbound2}
  Given a set $A=\{a_1=0, a_2,\cdots, a_n\}$ of integers such that $a_0< a_1< \cdots <a_n$ with $\gcd(a_2,\cdots, a_n)=1$,  the following conditions are equivalent:

\begin{enumerate}
  \item $|sA|=s(n-1)+1+ o(s)$
for infinitely many $s$, where $o(s)$ is an arithmetic function such that $\lim_{s\to\infty} o(s)=0$,
  \item $|sA|= s(n-1)+1$ for all $s\gg 0$,
  \item $|sA|= s(n-1)+1$ for some $s\ge 2$,
  \item $A=\{0,1,\cdots, n-1\}$,
  \item $|sA|= s(n-1)+1$ for all $s\ge 0$.
\end{enumerate}
\end{theorem}
\begin{proof}
By \propref{HilbertPol} we get that $(1)$ implies $(2)$.
On the other hand, $(2)$ trivially implies $(3)$.

Assume $(3)$, i.e. $|sA|= s(n-1)+1$ for some $s\ge 2$.
Notice that
$$
(s-1)A \cup \{(s-1) a_n +a_2,\cdots , (s-1) a_n +a_n\}\subset sA
$$
and, since $(s-1) a_n$ is the maximum of $(s-1)A$, we have
$$
(s-1)A \cap \{(s-1) a_n +a_2,\cdots , (s-1) a_n +a_n\}=\emptyset.
$$

By \propref{lowerbound} we have $|(s-1)A|\ge (s-1)(n-1)+1$, so
\begin{equation}\label{sA}
(s-1)A \cup \{(s-1) a_n +a_2,\cdots , (s-1) a_n +a_n\}= sA.
\end{equation}

We know that $\res[C_A]_s$ has as $\res$-basis the monomials
$t^{\alpha}w^s, \alpha\in sA$ and $X_1\res[C_A]_{s-1}$ is generated by
$t^{\alpha+a_1}w^s, \alpha\in (s-1)A$.
By \eqref{sA} we have that $(s-1)A+a_1\subset (s-1)A$ so the $\res$-vector space
$$
(\mathcal B_A)_s=\frac{\res[C_A]_s}{X_1\res[C_A]_{s-1}}
$$
is generated by the cosets of
$$
t^{(s-1)a_n+a_i}w^s, \quad i=2,\dots,n.
$$
This fact  implies that
$$
X_n^{s-1} (\mathcal B_A)_1 =(\mathcal B_A)_s
$$
Since the algebra $\mathcal B_A$ is standard
we get, multiplying both sides by $ (\mathcal B_A)_{(r-1)(s-1)}$, that
$$
X_n^{(s-1)r} (\mathcal B_A)_1= (\mathcal B_A)_{r(s-1)+1}
$$
for all $r\ge 1$.
Since $\dim_{\res}((\mathcal B_A)_t)=n-1 $, for $t\gg 0$ we obtain, \propref{multB},
$$
n-1\ge \dim_{\res}(\mathcal B_A)_{r(s-1)+1}= a_n
$$
for $r\gg 0$. Hence $a_n\le n-1 $ and we get $(4)$.

  The remaining implications are  easy computations.
\end{proof}

\begin{remark}
  \label{CRN}
The  curve $C_A$ for $A=\{0,\cdots, n-1\}$ is the rational normal curve of $\mathbb P^{n-1}_{\res}$, i.e. the curve defined by
$(u,v)\mapsto (u^{n-1},u^{n-2}v,\cdots, uv^{n-2},v^{n-1})$.
\end{remark}

\begin{remark}
From \lemref{1Reduction} (1) we get for a general set $A$ that
$|sA|= s(n-1)+1$ for all $s\ge 0$ if and only if $A$ is a $n$-term arithmetic progression, i.e.
$A=q_0+ q_1 [0,\cdots, n-1]$ for $q_0\in \mathbb N$ and $q_1\in \mathbb N\setminus \{0\}$.
\end{remark}

Next we use a result on additive combinatorics in order to improve an upper bound  of
the Castelnuovo-Mumford regularity of rational projective curves.
We first recall  the following result of V.F. Lev:

%\medskip
\begin{proposition}\cite[Theorem 1]{Lev96}
\label{lev}
Given  $A=\{a_1=0, a_2,\cdots, a_n\}$ with $\gcd(a_2,\cdots, a_n)=1$, it holds:
$$
|sA|-|(s-1)A|\ge \min\{a_n, s(n-2)+1\}
$$
for all $s\ge 2$.
\end{proposition}

In the following result we improve \cite[Theorem 2.7]{BGG17}, see also \cite{Lam21}, where an upper bound of
the Castelnuovo-Mumford regularity is given for a monomial projective curve $C_A$  under the hypothesis that $A$ is an arithmetic sequence.
We know that
$$\HF_{\mathcal B_A}(s)=\HF_A(s)-\HF_A(s-1)=|sA|-|(s-1)A|$$ so last result shows that the Hilbert function of the one-dimensional graded algebra
$\mathcal B_A$ grows rapidly.
This is the key point in the proof of the following result where we assume that $\res[C_A]$ is Cohen-Macaulay.
See \cite{CN83} and \cite{HS19} for several criteria
implying the Cohen-Macaulayness of $\res[C_A]$.

%%\medskip
\begin{theorem}
\label{bermejo}
Given  $A=\{a_1=0, a_2,\cdots, a_n\}$ with $\gcd(a_2,\cdots, a_n)=1$.
If the two-dimensional ring $\res[C_A]$ is Cohen-Macaulay then
$$
reg(\res[C_A]) \le \ceil[\bigg]{\frac{a_n-1}{n-2}}
$$
\end{theorem}
\begin{proof}
We write $s_0=\ceil[]{\frac{a_n-1}{n-2}}$.
Since $\res[C_A]$ is Cohen-Macaulay we have
$r(C_A)+1=reg(\res[C_A])$ and that $\mathcal B_A$ is a one-dimensional Cohen-Macaulay ring.
Hence we have
$$
\HF_{\mathcal B_A}(s)\le a_n
$$
for all $s\ge 1$, \cite[Chapter XII]{Mat77}.
From this inequality and \propref{lev} we get
$$
s(n-2)+1\le \HF_{\mathcal B_A}(s)\le \min\left\{a_n, \binom{s+n-2}{s}\right\}
$$
for $s=1,\cdots, s_0-1$; and
$$
\HF_{\mathcal B_A}(s)=a_n
$$
for $s\ge s_0$, i.e. $r(\mathcal B_A)\le s_0$.
Since
$r(C_A)+1 =r(\mathcal B_A)$ we get the claim:
$$
reg(\res[C_A])=r(C_A)+1 =r(\mathcal B_A)\le s_0.
$$
\end{proof}

\begin{example}[Macaulay's example]
\label{mac-ex}
In this example we consider the example of a non-singular, non-Cohen-Macaulay monomial projective curve given by Macaulay, \cite{Mac16}.
In this case the set is $A=\{0,1,3,4\}$.
The monomial curve $C_A$ associated to $A$ is defined by the  parameterization
$(u,v)\mapsto (u^4,u^3v,uv^3,v^4)$.
A computation with   Singular  \cite{DGPS}
give us that $\HF_A=\{1,4,9,13,17,21,\cdots\}$ and
$\HP_A(s)=4 s +1$.
Since the points $P_1, P_2$ are non-singular points of $C_A$, we deduce last identity from \propref{HilbertPol} as well.
\end{example}

\begin{example}
 \label{CMcurve}
We consider a especial case of \cite[Case A]{Lam21}.
Let us consider the set $A=\{0,7,8,9,10\}$.
From \cite[Theorem 2.1]{Lam21} we know that $\res[C_A]$ is Cohen-Macaulay and that
$r(C_A)=5$ that agrees with the upper bound of the \thmref{bermejo}.
The defining ideal of $C_A$ is minimally generated by:
$x_3^2- x_2  x_4,  x_3  x_4- x_2  x_5, x_4^2- x_3  x_5,  x_2^4- x_1  x_3  x_5^2, x_2^3  x_3- x_1  x_4  x_5^2,
 x_2^3  x_4- x_1  x_5^3.$
A straight computation shows
$$\HF_A=\{|sA|, s=0,1,\cdots \}=\{1,5, 12, 22, 32, 42, 52, 62, 72, \cdots\}$%%
$ and the Hilbert polynomial of $C_A$ is $\HP_A=10s-8$.
\end{example}

%%%%%%%%%%%%%%%%%%%%%%%%%%%%%%%%%%%%%%%%%%%%%%%%%%%%%%%%%%%%%%%%%%%%%%%%%%%%%%%%%%%%%%%%%%%%%%%%
%%%%%%%%%%%%%%%%%%%%%%%%%%%%%%%%%%%%%%%%%%%%%%%%%%%%%%%%%%%%%%%%%%%%%%%%%%%%%%%%%%%%%%%%%%%%%%%%
%%%%%%%%%%%%%%%%%%%%%%%%%%%%%%%%%%%%%%%%%%%%%%%%%%%%%%%%%%%%%%%%%%%%%%%%%%%%%%%%%%%%%%%%%%%%%%%%
%%%%%%%%%%%%%%%%%%%%%%%%%%%%%%%%%%%%%%%%%%%%%%%%%%%%%%%%%%%%%%%%%%%%%%%%%%%%%%%%%%%%%%%%%%%%%%%%
%%%%%%%%%%%%%%%%%%%%%%%%%%%%%%%%%%%%%%%%%%%%%%%%%%%%%%%%%%%%%%%%%%%%%%%%%%%%%%%%%%%%%%%%%%%%%%%%
%%%%%%%%%%%%%%%%%%%%%%%%%%%%%%%%%%%%%%%%%%%%%%%%%%%%%%%%%%%%%%%%%%%%%%%%%%%%%%%%%%%%%%%%%%%%%%%%

\bibliographystyle{amsplain}

\end{document}